\documentclass[10pt]{article}
\usepackage{pdflscape}
\usepackage{graphicx}
\usepackage{amsfonts}
\usepackage{amsthm}
\usepackage{lscape}
\usepackage{latexsym,amsmath}
\usepackage{amssymb,array}
\makeatletter 
\newtheorem{t1}{Theorem}[section]
\newtheorem{theorem}{Theorem}[section]

\newtheorem{defn}{Definition}[section]
\newtheorem{r1}{Remark}[section]
\newtheorem{remark}{Remark}

\begin{document}
\title{\textbf{A Natural Discrete One Parameter Polynomial Exponential Distribution}}
\date{}
\author{Sudhansu S. Maiti\footnote{Corresponding author. e-mail: dssm1@rediffmail.com}\\Department of Statistics, Visva-Bharati
University\\Santiniketan-731 235, West Bengal, India\and Molay Kumar Ruidas\\Department of Statistics, Trivenidevi Bhalotia College\\Kazi Nazrul University, Raniganj-713 347, India\and Sumanta Adhya\\Department of Statistics, West Bengal State University\\Barasat-700 126, India}
\maketitle
\begin{abstract}
In this paper, a new natural discrete version of the one parameter polynomial exponential family of distributions have been proposed and studied. The distribution is named as Natural Discrete One Parameter Polynomial Exponential (NDOPPE) distribution. Structural and reliability properties have been studied. Estimation procedure of the parameter of the distribution have been mentioned. Compound NDOPPE distribution in the context of collective risk model have been obtained in closed form. The new compound distribution has been compared with the classical compound Poisson, compound Negative binomial, compound discrete Lindley, compound xgamma-I and compound xgamma-II distributions regarding suitability of modelling extreme data with the help of some automobile claim.
\end{abstract}
\section{Introduction}
\par Data science is gaining momentum in recent years for analysing data. Fitting of an appropriate probability model to data is challenging and important aspect in this context. For quantitative data, it can be of continuous as well as count type. In life testing experiments, a number of continuous models have been suggested and studied, see, e.g., Lawless (2003) and Sinha(1986).
\par In practice, situations may arise where lifetime of a device is expressed in term of a count and may be considered as a discrete random variable. The number of motions of a pendulum before resting, the bulb in Xerox machine lights up each time a copy is taken, the number of times devices are switched on/off, the number of days a patient stays in a hospital, the number of weeks/months/years a kidney patient survives after treatment, the number of current fluctuations which an electrical item can withstand before its failure, etc. form discrete lifetimes.
\par Some standard discrete distributions, like, the binomial, Poisson, geometric and negative binomial distribution have been used to model lifetime (count) data by Barlow and Proschan (1965). These models are too restrictive. For example, the Poisson model is not appropriate as it imposes the restriction of equi-dispersion and the binomial model imposes the restriction of under-dispersion.
\par As a consequence, various models have been prescribed using discrete concentration and discrete analog approaches that are less restrictive (see, Nakagawa and Osaki (1975), Famoye (1993), among others). Some recent discrete distributions are due to Stein and Dattero (1984), Roy (2002,2003,2004), Krishna and Pundir (2009), Jazi et al. (2010), Gomez-Deniz (2010), Gomez-Deniz and Calderin-Ojeda (2011), Bakouch et al. (2014), Maiti et al. (2018), among others. Though there are a number of less restrictive discrete distributions, still there is room for constructing more flexible discrete lifetime distributions to suit for analysing various types of count data.
\par Bouchahed and Zeghdoudi (2018) has proposed a continuous distribution, called the one parameter polynomial exponential (OPPE)distribution, that is an unified approach in generalizing the Lindley's distribution [c.f., Lindley (1958)]. The Lindley distribution is considered to be more flexible than the exponential distribution, a popular lifetime model used in industry circle [c.f., Ghitaney et al (2008)]. The probability density function (pdf) of the OPPE random variable is given by
\begin{equation}\label{oppepdf1}
f_{X}(x,\theta)=h(\theta)p(x)\exp(-\theta x), \theta > 0,~ x > 0,
\end{equation}
where, $h(\theta)=\frac{1}{\sum_{k=0}^{r}a_{k}\frac{\Gamma(k+1)}{\theta^{k+1}}}$ ,
$p(x)={\sum_{k=0}^{r}a_{k}{x^k}}.$\\
The pdf can also be written as
\begin{equation}\label{oppepdf2}
f_{X}(x,\theta)=h(\theta){\sum_{k=0}^{r}a_{k}x^{k}\exp(-\theta x)}=\frac{{\sum_{k=0}^{r}a_{k}\frac{\Gamma(k+1)}{\theta^{k+1}}f_{GA}(x;k+1,\theta)}}{\sum_{k=0}^{r}a_{k}\frac{\Gamma(k+1)}{\theta^{k+1}}},
\end{equation}
where $f_{GA}(x;k+1,\theta)$ is the pdf of a gamma distribution with shape parameter $(k+1)$ and scale parameter $\theta,$ and $a_k$'s are known non-negative constants and $r$ is a known positive integer. The distribution is a finite mixture of $(r+1)$ gamma distributions.\\
The cumulative distribution function (cdf) of the random variable is
\begin{equation}\label{oppecdf}
F(t)=\frac{{\sum_{k=0}^{r}a_{k}\frac{\Gamma{(k+1)}}{\theta^{k+1}}\Gamma{(k+1,\theta t)}}}{{\sum_{k=0}^{r}a_{k}\frac{\Gamma(k+1)}{\theta^{k+1}}}}, \theta > 0 ,~ t > 0,
\end{equation}
where $\Gamma(m,t)=\frac{1}{\Gamma(m)}\int_{0}^{t}\exp(-x)x^{m-1}dx.$
\par A random variable X is said to have a natural discrete one parameter polynomial exponential (NDOPPE) distribution with success probability $\theta$, if its probability mass function (pmf) is given by
\begin{eqnarray}\label{ndoppepmf1}
 p(x;\theta)&=&\frac{\sum_{k=0}^{r}a_{k}\frac{\Gamma(k+1)}{\theta^{k+1}}f_{NB}(x;k+1,~\theta)}{\sum_{k=0}^{r}a_{k}\frac{\Gamma(k+1)}{\theta^{k+1}}},~ x=0,~1,~2,~...,~0<\theta<1
 \nonumber \\
 &=&\frac{\sum_{k=0}^{r}a_{k}\frac{\Gamma(k+1)}{\theta^{k+1}}\binom {x+k}{x}(1-\theta)^x \theta^{k+1}}{\sum_{k=0}^{r}a_{k}\frac{\Gamma(k+1)}{\theta^{k+1}}},
\end{eqnarray}
where $f_{NB}(x;k+1,\theta)$ is the pmf of a negative binomial distribution with number of success $(k+1)$ and success probability $\theta,$ and $a_k$'s are known non-negative constants and $r$ is a known positive integer. The distribution is a finite mixture of $(r+1)$ negative binomial distributions.\\
The pmf can also be expressed as
\begin{equation}\label{ndoppepmf2}
p(x;\theta)=h(\theta)p(x)(1-\theta)^x,~0<\theta<1,~ x = 0,~1,~2,~...
\end{equation}
where, $h(\theta)=\frac{1}{\sum_{k=0}^{r}a_{k}\frac{\Gamma(k+1)}{\theta^{k+1}}}$ ,
$p(x)={\sum_{k=0}^{r}a_{k}\Gamma(k+1)\binom {x+k}{x}}.$\\
The cdf of the random variable is
\begin{equation}\label{ndoppecdf}
F(x;\theta)=\frac{\sum_{k=0}^{r}a_{k}\frac{\Gamma(k+1)}{\theta^{k+1}}I_\theta(x,k+2)}{\sum_{k=0}^{r}a_{k}\frac{\Gamma(k+1)}{\theta^{k+1}}},  
\end{equation}
where, $I_p(m,~n)=\frac{1}{B(m,~n)}\int_0^px^{m-1}(1-x)^{n-1}dx.$
\begin{r1}
It is to be noted that $ p(0;\theta)=\frac{\sum_{k=0}^{r}a_{k}\Gamma(k+1)}{\sum_{k=0}^{r}a_{k}\frac{\Gamma(k+1)}{\theta^{k+1}}}$, and the other probabilities can be calculated recursively with the relationship $ p(x+1;\theta)=\frac{1-\theta}{1+x}\frac{\sum_{k=0}^{r}a_{k}(x+k+1)!}{\sum_{k=0}^{r}a_{k}(x+k)!}p(x;\theta)$.
\end{r1}
\begin{r1}
1. If $r=0,~a_0=1$, then ($\ref{ndoppepmf2}$) reduces to the geometric distribution.\\
2. If $r=1,~a_0=1,~a_1=1$, then ($\ref{ndoppepmf2}$) reduces to the Natural Discrete Lindley (NDL) distribution [c.f. Ahmed and Afify (2019)].
\end{r1}
\section{Moments and related Measures}
\subsection{Moment generating function}
The moment generating function of the NDOPPE distribution is given by 
 \begin{eqnarray*}
 M_x(t)&=& E(e^{tX})
 \nonumber\\
 &=& \frac{\sum_{k=0}^{r}a_{k}\frac{\Gamma(k+1)}{(1-\bar{\theta} e^t)^{k+1}}}{\sum_{k=0}^{r}a_{k}\frac{\Gamma(k+1)}{\theta^{k+1}}},~\bar{\theta}=1-\theta,~t<-\ln{(1-\theta)}
 \end{eqnarray*}
\subsection{Probability generating function}
The probabilty generating function of the NDOPPE distribution is given by 
 \begin{eqnarray*}
 P_x(s)&=& E(s^X)
 \nonumber\\
 &=& \frac{\sum_{k=0}^{r}a_{k}\frac{\Gamma(k+1)}{(1-\bar{\theta} s)^{k+1}}}{\sum_{k=0}^{r}a_{k}\frac{\Gamma(k+1)}{\theta^{k+1}}},~\mid s\mid<\frac{1}{1-\theta}
 \end{eqnarray*}
\subsection{Characteristics function}
The characteristic function of the NDOPPE distribution is given by 
 \begin{eqnarray*}
 \phi(t)&=&E(e^{itX})
 \nonumber\\
 &=&\frac{\sum_{k=0}^{r}a_{k}\frac{\Gamma(k+1)}{(1-\bar{\theta} e^{it})^{k+1}}}{\sum_{k=0}^{r}a_{k}\frac{\Gamma(k+1)}{\theta^{k+1}}}
 \end{eqnarray*}
\subsection{Cumulant generating function}
The cumulant generating function of the NDOPPE distribution is given by 
 \begin{eqnarray*}
 \kappa(t)&=&\ln M_x(t)
 \nonumber\\
 &=&-\ln \left({\sum_{k=0}^{r}a_{k}\frac{\Gamma(k+1)}{(1-\bar{\theta} e^{it})^{k+1}}}\right)+\ln \left(\sum_{k=0}^{r}a_{k}\frac{\Gamma(k+1)}{\theta^{k+1}}\right)
 \end{eqnarray*}
\subsection{Raw moments}
The raw moments  of the NDOPPE distribution is given by 
 \begin{eqnarray*}
 \mu'_1&=&\frac{\sum_{k=0}^{r}a_{k}\frac{\Gamma(k+1)}{\theta^{k+1}}\left(\frac{(k+1)\bar{\theta}}{\theta}\right)}{\sum_{k=0}^{r}a_{k}\frac{\Gamma(k+1)}{\theta^{k+1}}}
 \nonumber\\
 \mu'_2&=&\frac{\sum_{k=0}^{r}a_{k}\frac{\Gamma(k+1)}{\theta^{k+1}}\left(\frac{(k+1)\bar{\theta}(1+(k+1)\bar{\theta})}{\theta^2}\right)}{\sum_{k=0}^{r}a_{k}\frac{\Gamma(k+1)}{\theta^{k+1}}}
 \nonumber\\
 \mu'_3&=&\frac{\sum_{k=0}^{r}a_{k}\frac{\Gamma(k+1)}{\theta^{k+1}}\left(\frac{\bar{\theta}((k+1)\theta^2+3(k+1)\theta\bar{\theta}+k(k+1)\bar{\theta}^2)}{\theta^3}\right)}{\sum_{k=0}^{r}a_{k}\frac{\Gamma(k+1)}{\theta^{k+1}}}
 \nonumber\\
 \mu'_4&=&\frac{\sum_{k=0}^{r}a_{k}\frac{\Gamma(k+1)}{\theta^{k+1}}\left(\frac{\bar{\theta}((k+1)\theta^3+7(k+1)\theta^2\bar{\theta}+6k(k+1)\theta\bar{\theta}^2+(k-1)k(k+1)\bar{\theta}^3}{\theta^4}\right)}{\sum_{k=0}^{r}a_{k}\frac{\Gamma(k+1)}{\theta^{k+1}}}
 \end{eqnarray*}
 
 
 The corresponding Index of Dispersion (ID) are :\\
 \begin{eqnarray*}
 ID(x) &=& \frac{Var(X)}{E(X)}
 \nonumber\\
 &=& \frac{\sum_{k=0}^{r}a_{k}\frac{\Gamma(k+1)}{\theta^{k+1}}\left(\frac{(k+1)\bar{\theta}(1+(k+1)\bar{\theta})}{\theta^2}\right)}{\sum_{k=0}^{r}a_{k}\frac{\Gamma(k+1)}{\theta^{k+1}}\left(\frac{(k+1)\bar{\theta}}{\theta}\right)}-\frac{\sum_{k=0}^{r}a_{k}\frac{\Gamma(k+1)}{\theta^{k+1}}\left(\frac{(k+1)\bar{\theta}}{\theta}\right)}{\sum_{k=0}^{r}a_{k}\frac{\Gamma(k+1)}{\theta^{k+1}}}
 \end{eqnarray*}
 The ID indicates whether a certain distribution is suitable for under or over dispersed data sets, and has application in ecology for measuring clustering(see, e.g, Johnson, 1992). If ID $\geqslant 1$, the distribution is over dispersed $\forall \theta $. We note that the ID decreases monotonically in $\theta$. ID $\rightarrow 1$  as $\theta \rightarrow 1$, while ID $\rightarrow \infty $  as $\theta \rightarrow 0$ . So, the NDOPPE distribution should only be used in the count data analysis with over dispersion.
\subsection{Factorial moments}
The m-th factorial moment  of the NDOPPE distribution is given by
\begin{eqnarray*}
 \mu_{(m)}^{'}&=&E[X(X-1)(X-2)...(X-m+1)]=h(\theta)\left(\frac{\bar{\theta}}{\theta}\right)^m\sum_{k=0}^ra_k\frac{\Gamma(m+k+1)}{\theta^{k+1}},~m=1,~2,~....
 \end{eqnarray*}
\section{Stochastic ordering}
Stochastic orders are important measures to judge comparative behaviours of random variables.
\begin{defn} 
Let X and Y are the two random variables with cumulative distribution functions $F_X(.)$ and $F_Y(.)$ respectively. Then X is said to be smaller than Y in the 
\begin{itemize}
\item Stochastic order $(X\prec_sY), if F_X(t)\geq F_Y(t), \forall$t
\item Convex order $(X\prec_{cx}Y)$ if for all convex functions $\phi$ and provided expectation exist,   $E[\phi(X)\leq E[\phi(Y)]$.
\item Hazard rate order $(X\prec_{hr}Y) ,if h_X(t)\geq h_Y(t), \forall$t.
\item Likelihood Ratio order $(X\prec_{lr}Y) ,if  \frac{f_X(t)}{f_Y(t)}$ is decreasing in t.
\end{itemize}
\end{defn}
\begin{remark}
Likelihood ratio order $\Rightarrow$ Hazard rate order $\Rightarrow$ Stochastic order if E[X] = E[Y] , then Convex order $\iff$ Stochastic order.
\end{remark}
\begin{theorem}
Let $X_i \sim NDOPPE (\theta_i), i=1,2$ be two random variables. If $\theta_1\leq \theta_2, then X_1 \prec_{lr}X_2,X_1 \prec_{hr}X_2,X_1 \prec_s X_2$ and $X_1 \prec_{cx}X_2$
\end{theorem}
\begin{proof}
we have 
\begin{eqnarray*}
L(t)=\frac{f_{X_1}(t)}{f_{X_2}(t)}=\frac{h(\theta_1)p(t)(1-\theta_1)^t}{h(\theta_2)p(t)(1-\theta_2)^t}=\frac{h(\theta_1)(1-\theta_1)^t}{h(\theta_2)(1-\theta_2)^t}
\end{eqnarray*}
Clearly, it is evident that  $L(t+1)\geq L(t), \forall \theta_1 \leq \theta_2 $
\end{proof}
\section{Reliability Properties}
The reliability function of the NDOPPE distribution is given by
\begin{eqnarray}
R(t;\theta)&=&P(X\geq t) \nonumber\\&=&\frac{\sum_{k=0}^{r}a_{k}\frac{\Gamma(k+1)}{\theta^{k+1}}I_\theta(k+2,t)}{\sum_{k=0}^{r}a_{k}\frac{\Gamma(k+1)}{\theta^{k+1}}}
\end{eqnarray}
and its failure rate function is given by
\begin{eqnarray}
r(t;\theta)&=&\frac{p(t;\theta)}{R(t;\theta)}\nonumber\\&=&\frac{\sum_{k=0}^{r}a_{k}\frac{\Gamma(k+1)}{\theta^{k+1}}\binom {t+k}{t}(1-\theta)^t \theta^{k+1}}{\sum_{k=0}^{r}a_{k}\frac{\Gamma(k+1)}{\theta^{k+1}}I_\theta(k+2,t)}
\end{eqnarray}
Note that $\frac{\sum_{k=0}^{r}a_{k}\Gamma(k+1)}{\sum_{k=0}^{r}a_{k}\frac{\Gamma(k+1)}{\theta^{k+1}}}<r(t;\theta)<\theta$ for $\forall t.$\\
\subsection{Stress Strength Reliability}
The stress-strength parameter (R) plays and important role in the reliability analysis as it measures the system performance. Moreover, R provides the probability of a system failure, the system fails whenever the the applied stress is greater than its strength, i.e $R=P(X<Y)$. Here $ Y\sim  DGL(\theta_1)$ denotes the strength of a system subject to stress $X$, $ X \sim  DGL(\theta_2)$ , X and Y are independent of each other. In our case, the stress-strength parameter is given by 
\begin{eqnarray*}
 R&=&\sum_{y=0}^{\infty}P(X \leq Y | Y=y)p_Y(y)
  \nonumber\\
 &=&\sum_{y=0}^{\infty}F_X(y)p_Y(y)
  \nonumber\\
 &=&\sum_{k=0}^{r}\sum_{l=0}^{r}a_{k}a_{l}\frac{\Gamma (k+1)}{\theta_2^{k+1}}\frac{\Gamma (l+1)}{\theta_1^{l+1}}\sum_{y=0}^{\infty}I_{\theta_2}(y,k+2)\binom{y+l}y \theta_1^{l+1}(1-\theta_1)^y
\end{eqnarray*}

\section{Distribution of maximum and minimum in a random sample from NDOPPE distribution }
Maximum and minimum of random variables arise in reliability. Let $X_i, i=1,2,..,n$ be iid random variables from the NDOPPE distribution with parameter $\theta$ . Then ,the cdf of minimum, $Min(X_1,X_2,..X_n)$ and maximum, $Max(X_1,X_2,..X_n)$ are given by
\begin{eqnarray*}
 F_{X_{(1)}}(x)&=&1-\left[1-\frac{\sum_{k=0}^{r}a_{k}\frac{\Gamma(k+1)}{\theta^{k+1}}I_\theta(x,k+2)}{\sum_{k=0}^{r}a_{k}\frac{\Gamma(k+1)}{\theta^{k+1}}} \right]^n
 \nonumber\\
  F_{X_{(n)}}(x)&=&\left[\frac{\sum_{k=0}^{r}a_{k}\frac{\Gamma(k+1)}{\theta^{k+1}}I_\theta(x,k+2)}{\sum_{k=0}^{r}a_{k}\frac{\Gamma(k+1)}{\theta^{k+1}}} \right]^n
\end{eqnarray*}
\section{Estimation of parameter}
Let $X_1,X_2,..,X_n$ be a random sample of size $n$ drawn from the NDOPPE distribution. To derive the Maximum Likelihood Estimator (MLE) of $\theta$, the log-likelihood function, ln$l(x_i;\theta)$ is given by :
\begin{eqnarray*}
ln l(x_i;\theta)=nln h(\theta)+\sum_{i=1}^{n}ln p(x_i)+\sum_{i=1}^{n}x_i ln(1-\theta)
\end{eqnarray*}
The derivative of ln $l(x_i;\theta)$ with respect to $\theta$ is :
\begin{eqnarray*}
 \frac{dln l(x_i;\theta)}{d\theta}=\frac{nh'(\theta)}{h(\theta)}-\frac{1}{(1-\theta)}\sum_{i=1}^{n}x_i
\end{eqnarray*}
For the NDOPPE Distribution , the method of moment (MOM) and the ML estimator of the parameter $\theta$ are the same and it can be obtained by solving the following non-linear equation 
\begin{eqnarray*}
 \frac{(1-\theta)h'(\theta)}{h(\theta)}- \bar{x}=0
\end{eqnarray*}
Since the pmf in ($\ref{ndoppepmf2}$) belongs to the one parameter exponential family, $\bar{X}$ is the MVUE of $\frac{(1-\theta)h'(\theta)}{h(\theta)}$, but exclusively, the MVUE of $\theta$ is not available.
\section{Comparing with the Poisson, Negative binomial, Discrete Lindley and dxgamma distributions in the collective risk model}
In the collective risk theory, the random variable of interest is the aggregate claim defined by $S =\sum_{i=1}^N X_i$,
where $N$ is the random variable denoting the number of claims and $X_i$ , for $i = 1, 2, . . .$ is the random variable denoting size or amount of the $i$th claim. By assuming that $X_1,X_2, . . .,$ are
independent and identically distributed random variables which are also independent of the random
number of claims $N$, the pdf of the sum $S$ is given by
$f_S(x) =\sum_{n=0}^{\infty}p_nf^{n^*}(x)$, where $p_n$ denotes the probability of $n$ claims (primary distribution) and
$f^{n^*}(x)$is the $n$th fold convolution of $f(x)$, the pdf of the claim amount (secondary distribution). The mgf of $S$ is given by, $M_S(t)=M_N(\ln M_X(t))$. It is also well known that under the model assumed, the mean and variance of the random aggregate claim amount are $E(S)=E(N).E(X)$ and $Var(S)=Var(X)E(N)+E^2(X)Var(N).$ [for details, see, Bowers et al. (1997)]\\
Large claims in insurance play a very important role, mainly if it is considered in relation to reinsurance. In the framework of reinsurance premiums, a reinsurer needs to be sure not to use a
distribution whose tail fades away to zero too quickly and will be tilted toward a heavy-tailed distribution. For this purpose, the Pareto and log-normal
distributions have been commonly used in reinsurance premium computation. Nevertheless, the compound Poisson (sometimes negative binomial) model has been traditionally considered when the size of a single claim is modelled by an exponential distribution, mainly for the complexity of the collective risk model under these kinds of distributions. G$\acute{o}$mez-D$\acute{e}$niz and Calder$\acute{i}$n-Ojeda (2011) developed models by
considering the discrete Lindley distributions. In this section, some alternative models are developed by
considering the discrete xgamma distributions proposed in the previous sections.\\
The next results show that a closed form expression is obtained when a natural discrete one parameter polynomial exponential and exponential
distribution are assumed as primary and secondary distributions, respectively.
\begin{t1}\label{thcndoppe}
If we assume a natural discrete one parameter polynomial exponential (NDOPPE) distribution with parameter $0 < \theta < 1$ as primary
distribution and an exponential distribution with parameter $\gamma > 0$ as secondary distribution,
then the pdf of the random variable $S =\sum_{i=1}^N X_i$ is given by
\begin{eqnarray}\label{cndoppe}
f_S(x)&=&\gamma (1-\theta)e^{-\gamma x}\frac{\sum_{k=0}^{r}a_{k}\Gamma(k+2){_1F_1(k+2;2;\gamma (1-\theta)x)}}{\sum_{k=0}^{r}a_{k}\frac{\Gamma(k+1)}{\theta^{k+1}}}~~~~~~~~~~~~~~~~~~~~\mbox{for}~x>0\nonumber\\
&=& \frac{\sum_{k=0}^{r}a_{k}\Gamma(k+1)}{\sum_{k=0}^{r}a_{k}\frac{\Gamma(k+1)}{\theta^{k+1}}}~~~~~\mbox{for}~x=0,
\end{eqnarray}
where, $_1F_1(.;.;.)$ is the confluent hypergeometric function.
\end{t1}
{\it Proof:} By assuming that the claim amount follows an exponential distribution with parameter
$\gamma > 0$, the $n$th fold convolution of exponential distribution has a closed form and it is given
by
$$f^{n^*}(x)=\frac{\gamma^n}{(n-1)!}e^{-\gamma x},~x>0,~n=1,~2,~...$$
that is, it is a gamma distribution with parameters $n$ and $\gamma$. If the primary distribution is a negative binomial with number of success $(k+1),~k$ being non-negative integer and success probability $\theta$, then the pdf of the random variable total claim amount [see Rolski et al. (1999)] is given by
\begin{eqnarray*}
f_S(x)&=&(k+1)\gamma (1-\theta)\theta^{k+1}e^{-\gamma x}{_1F_1(k+2;2;\gamma (1-\theta)x)}~~~~~~~~~~~~~~~~~~~~\mbox{for}~x>0\nonumber\\
&=& \theta^{k+1}~~~~~\mbox{for}~x=0,
\end{eqnarray*} and after algebraic computations the theorem holds.~~~~~~~~~$\square$\\
The mgf of $S$ is, 
\begin{eqnarray*}
M_S(t)&=&\frac{\sum_{k=0}^{r}a_{k}\frac{\Gamma(k+1)}{\left[1-\bar{\theta}\left(1-\frac{t}{\gamma}\right)^{-1}\right]^{k+1}}}{\sum_{k=0}^{r}a_{k}\frac{\Gamma(k+1)}{\theta^{k+1}}},~t<\gamma\theta.
\end{eqnarray*}
Hence Mean, $$E(S)=\frac{\left(\frac{\bar{\theta}}{\theta}\right)\sum_{k=0}^ra_k\frac{\Gamma(k+2)}{\theta^{k+1}}}{\gamma\sum_{k=0}^ra_k\frac{\Gamma(k+1)}{\theta^{k+1}}}$$ and Variance,$$Var(S)=\frac{1}{\gamma^2}\left[\frac{\left(\frac{\bar{\theta}}{\theta}\right)^2\sum_{k=0}^ra_k\frac{\Gamma(k+3)}{\theta^{k+1}}}{\sum_{k=0}^ra_k\frac{\Gamma(k+1)}{\theta^{k+1}}}-\frac{\left(\frac{\bar{\theta}}{\theta}\right)\sum_{k=0}^ra_k\frac{\Gamma(k+2)}{\theta^{k+1}}}{\sum_{k=0}^ra_k\frac{\Gamma(k+1)}{\theta^{k+1}}}\left\{\frac{\left(\frac{\bar{\theta}}{\theta}\right)\sum_{k=0}^ra_k\frac{\Gamma(k+2)}{\theta^{k+1}}}{\sum_{k=0}^ra_k\frac{\Gamma(k+1)}{\theta^{k+1}}}-2\right\}\right].$$
\begin{t1}\label{thcxg1}[Maiti et al.(2018)]
If we assume a discrete xgamma-I distribution with parameter $0 < p < 1$ as primary
distribution and an exponential distribution with parameter $\gamma > 0$ as secondary distribution,
then the pdf of the random variable $S =\sum_{i=1}^N X_i$ is given by
\begin{eqnarray}\label{cxg1}
f_S(x)&=&\frac{\gamma p}{1-\ln{p}}\left[1-(2-3p)\ln{p}+\frac{1-4p}{2}(\ln{p})^2+\ln{p}\left(p-1+\frac{3-5p}{2}\ln{p}\right)\gamma p x\right.\nonumber\\&&\left.+(1-p)\frac{(\ln{p})^2}{2}(\gamma p x)^2\right]e^{-\gamma(1-p)x}~~~~~~~~~~~~~~~~~~~~\mbox{for}~x>0\nonumber\\
&=& \frac{1}{1-\ln{p}}\left[1-\ln{p}-p\left\{1-2\ln{p}+\frac{(\ln{p})^2}{2}\right\}\right]~~~~~\mbox{for}~x=0
\end{eqnarray}
\end{t1}
The mgf of $S$ is, 
\begin{eqnarray*}
M_S(t)&=&a_1\frac{1-p}{1-p\left(1-\frac{t}{\gamma}\right)^{-1}}+b_1\frac{(1-p)^2}{\left\{1-p\left(1-\frac{t}{\gamma}\right)^{-1}\right\}^2}+c_1\frac{(1-p)^3}{\left\{1-p\left(1-\frac{t}{\gamma}\right)^{-1}\right\}^3},~t<\gamma (1-p).
\end{eqnarray*}
Hence Mean, $E(S)=\left[a_1+2b_1+3c_1\right]\frac{p}{\gamma(1-p)}$ and Variance, $Var(S)=\frac{p}{\gamma^2(1-p)^2}[(1-p)(a_1+2b_1+3c_1)+(1+p)a_1+2(1+2p)b_1+3(1+3p)c_1-p(a_1+2b_1+3c_1)^2].$
\begin{t1}\label{thcxg2}[Maiti et al.(2018)]
If we assume a discrete xgamma-II distribution with parameter $0 < p < 1$ as primary
distribution and an exponential distribution with parameter $\gamma > 0$ as secondary distribution,
then the pdf of the random variable $S =\sum_{i=1}^N X_i$ is given by
\begin{eqnarray}\label{cxg2}
f_S(x)&=&\frac{2(1-p)^3\gamma p}{2(1-p)^2-p(1+p)\ln{p}}\left[1-\frac{\ln{p}}{2}-\frac{3\ln{p}}{2}(\gamma p x)-\frac{\ln{p}}{2}(\gamma p x)^2\right]e^{-\gamma(1-p)x}~\mbox{for}~x>0\nonumber\\
&=& \frac{2(1-p)^3}{2(1-p)^2-p(1+p)\ln{p}}~~~~~~~~~~~~~~~~~~~\mbox{for}~x=0
\end{eqnarray}
\end{t1}
The mgf of $S$ is,
\begin{eqnarray*}
M_S(t)&=&a_2\frac{1-p}{1-p\left(1-\frac{t}{\gamma}\right)^{-1}}+b_2\frac{(1-p)^2}{\left\{1-p\left(1-\frac{t}{\gamma}\right)^{-1}\right\}^2}+c_2\frac{(1-p)^3}{\left\{1-p\left(1-\frac{t}{\gamma}\right)^{-1}\right\}^3},~t<\gamma (1-p).
\end{eqnarray*}
Hence Mean, $E(S)=\left[a_2+2b_2+3c_2\right]\frac{p}{\gamma(1-p)}$ and Variance, $Var(S)=\frac{p}{\gamma^2(1-p)^2}[(1-p)(a_2+2b_2+3c_2)+(1+p)a_2+2(1+2p)b_2+3(1+3p)c_2-p(a_2+2b_2+3c_2)^2].$

\begin{t1}\label{thcxg3}[G$\acute{o}$mez-D$\acute{e}$niz and Calder$\acute{i}$n-Ojeda (2011)]
If we assume a discrete Lindley distribution with parameter $0 < \lambda< 1$ as primary
distribution and an exponential distribution with parameter $\gamma > 0$ as secondary distribution,
then the pdf of the random variable $S =\sum_{i=1}^N X_i$ is given by
\begin{eqnarray}\label{cxg3}
f_S(x)&=&\frac{\gamma\lambda(1-\lambda+(\lambda^2\gamma x+\lambda(3-\gamma x)-2)\ln\lambda)}{1-\ln\lambda}e^{-\gamma(1-\lambda)x}~\mbox{for}~x>0\nonumber\\
&=& \frac{1-\lambda+(2\lambda-1)\ln\lambda}{1-\ln\lambda}~~~~~~~~~~~~~~~~~~~\mbox{for}~x=0
\end{eqnarray}
\end{t1}
The mgf of $S$ is,
\begin{eqnarray*}
M_S(t)&=&\frac{1}{1-\ln\lambda}\left[\frac{2(1-\lambda+\lambda\ln\lambda)-\ln\lambda}{1-\lambda\left(1-\frac{t}{\gamma}\right)^{-1}}-\frac{e^t\lambda\ln\lambda}{\left\{1-p\left(1-\frac{t}{\gamma}\right)^{-1}\right\}^2}\right],~t<\gamma (1-\lambda).
\end{eqnarray*}
Hence Mean, $E(S)=\frac{\lambda\{1-\lambda+(\lambda-2)\ln\lambda\}}{\gamma(1-\lambda)^2(1-\ln\lambda)}$ and Variance, $Var(S)=\frac{\lambda}{\gamma^2(1-\lambda)^4(1-\ln\lambda)^2}[(1-\lambda)^2(1-\ln\lambda)\{1-\lambda+(\lambda-2)\ln\lambda\}+(1-\lambda)^2-(3-4\lambda+\lambda^2)\ln\lambda+(2-3\lambda)(\ln\lambda)^2].$\\
The most well-known aggregate claim model is the one obtained when the primary and secondary
distribution are the Poisson and exponential distribution, respectively [see, Rolski et al. (1999)]. In this case, the distribution of the random aggregate claim size is given by
\begin{eqnarray}\label{cxg4}
f_S(x)&=&\sqrt{\frac{\gamma\alpha}{x}}I_1(2\sqrt{\gamma\alpha x})e^{-(\alpha+\gamma x)}~\mbox{for}~x>0\nonumber\\
&=& e^{-\alpha}~~~~~~~~~~~~~~~~~~~\mbox{for}~x=0
\end{eqnarray}
Here, $\alpha > 0$ and $\gamma > 0$ are the parameters of the Poisson and exponential
distributions, respectively, and
$$I_v(z)=\sum_{k=0}^{\infty}\frac{(z/2)^{2k+v}}{\Gamma(k+1)\Gamma(v+k+1)},~z\in \mathbb{R},~v\in \mathbb{R},$$
represents the modified Bessel function of the first kind.\\
The mgf of $S$ is,
\begin{eqnarray*}
M_S(t)&=&e^{-\alpha\left\{1-\left(1-\frac{t}{\gamma}\right)^{-1}\right\}},~t<\gamma (1-\alpha).
\end{eqnarray*}
Hence Mean, $E(S)=\frac{\alpha}{\gamma}$ and Variance, $Var(S)=\frac{2\alpha}{\gamma^2}.$\\
Another well-known aggregate claim model is the one obtained when the primary and secondary
distribution are the negative binomial with parameters $r$ and $0<1-p<1$, and exponential distribution, respectively. In this case, the distribution of the random aggregate claim size is given by
\begin{eqnarray}\label{cxg5}
f_S(x)&=&\gamma r (1-p)^rpe^{-\gamma x} {_1F_1(1+r;2;\gamma px)}~\mbox{for}~x>0\nonumber\\
&=& (1-p)^r~~~~~~~~~~~~~~~~~~~\mbox{for}~x=0
\end{eqnarray}
The mgf of $S$ is,
\begin{eqnarray*}
M_S(t)&=&(1-p)^r\left\{1-p\left(1-\frac{t}{\gamma}\right)^{-1}\right\}^{-r},~t<\gamma (1-p).
\end{eqnarray*}
Hence Mean, $E(S)=\frac{rp}{\gamma (1-p)}$ and Variance, $Var(S)=\frac{rp(2-p)}{\gamma^2(1-p)^2}.$
\section{Application}
In this section, we have fitted eight data sets [four from Willmot(1987), two from Berm$\acute{u}$dez (2009), remaining  from Boucher et al. (2007)] and compared these with Poisson, Negative binomial, discrete Lindley, dxgamma-I, dxgamma-II and NDOPPE distributions. We have shown the fitted probabilities for each data set and comparison have been made in term of negative log-likelihood as well as observed chi-square statistic. Summarized results have been shown in Tables 1-8.\\
In all the data sets, there is high proportion of zero values observed. We have estimated the parameter of each model by the method of maximum likelihood. From the Tables 1-8, it is found that the family of NDOPPE distribution has better performance than the other distributions in both negative log-likelihood sense and for observed chi-square value.\\
Therefore, the model given in ($\ref{cndoppe}$) seems more suitable for empirical modelling in the framework of collective risk in the actuarial literature than the others discussed in section 7.
\section{Conclusion}
A natural discrete analog of the continuous one parameter polynomial exponential (OPPE) family of distributions, called the NDOPPE family, has been proposed. The structural and reliability properties of the distribution have been studied. Estimation procedures of the parameter have been mentioned. Compound NDOPPE distribution in the context of collective risk model have been obtained. The new compound discrete distribution has been compared with the classical compound Poisson, compound Negative binomial, compound discrete Lindley, compound discrete xgamma-I and compound discrete xgamma-II distributions with the help of some automobile claim data. It has been observed that the NDOPPE family of distributions is more suitable for empirical modelling in the framework of collective risk.\\
Some further properties of the distribution are under study. Estimation of the pmf and cdf, the reliability functions and their properties are in progress and will be communicated shortly.
\begin{center}{\bf Table:1 Data Set-I}\\
\vspace{0.2 in}
{
\footnotesize 
\begin{tabular}{|c|c|c|c|c|c|c|c|}
\hline
Number&Observed&Fitted& Fitted&Fitted&Fitted& Fitted& Fitted\\
of&&Poisson &Negative&dLindley &dxgamma-I &dxgamma-II&NDOPPE\\
Claims&& &Binomial& & & &$a_0=a_1=1$\\\hline
0&103704&102627.9&103217.2&103350.1&103321.3&103840.8&103519.4\\\hline
1&14075&15923.36&14861.67&14626.28&14607.95&13715.99&14339.05\\\hline
2&1766&1235.304&1604.886&1681.739& 1748.554&2046.422&1765.495\\\hline
3&255&63.8884&154.0523&175.7091&161.5036 &226.9247&203.7906\\\hline
4&45&2.478171&13.86321&17.36953&12.75092&20.96838&22.58254\\\hline
5&6& 0.07690074&1.197651 &1.656130&0.9115403&1.731256&2.432916\\\hline
6&2&0.001988605&0.1005918&0.1539398&0.06091039&0.1326196&0.2567596\\\hline
Negative&-&55108.46&54697.39&54659.61& 54678.22& 54754.24&54630.26\\
log-likelihood&&&&&&&\\\hline
Observed&-&4218.796&251.3145&139.4778&246.8612&388.6608&57.37906\\
Chisquare&&&&&&&\\\hline
\end{tabular}
}
\end{center}

\begin{center}{\bf Table:2 Data Set-II}\\
\vspace{0.2 in}
{
\footnotesize 
\begin{tabular}{|c|c|c|c|c|c|c|c|}
\hline
Number&Observed&Fitted& Fitted&Fitted&Fitted& Fitted& Fitted\\
of&&Poisson &Negative&dLindley &dxgamma-I &dxgamma-II&NDOPPE\\
Claims&& &Binomial& & & &$a_0=a_1=a_2=a_3=1$\\\hline
0&370412&369253.7&370786&371140.5&371168.6&372590.2&370651.6\\\hline
1&46545&48637.64&45826.83&45198.52&45126.06& 42689.26&46250.37\\\hline
2&3935&3203.244&4247.933&4463.65&4566.781&5413.719&4027.637\\\hline
3&317&140.6425&350.0120&400.3299&353.7008&504.7849&290.5628\\\hline
4&28&4.631312&27.03706&33.96222&23.35404&39.11710&18.63986\\\hline
5&3&0.1220061&2.004967& 2.778624&1.394727&2.706197&1.103071\\\hline
Negative&-&171373.2&171152.4&171196.2&171204.7&171890.5&171139.3\\
log-likelihood&&&&&&&\\\hline
Observed&-&667.7778&38.32639&122.5626&140.1415&1593.307&14.53022\\
Chisquare&&&&&&&\\\hline
\end{tabular}
}
\end{center}

\begin{center}{\bf Table:3 Data Set-III}\\
\vspace{0.2 in}
{
\footnotesize 
\begin{tabular}{|c|c|c|c|c|c|c|c|}
\hline
Number&Observed&Fitted& Fitted&Fitted&Fitted& Fitted& Fitted\\
of&&Poisson &Negative&dLindley &dxgamma-I &dxgamma-II&NDOPPE\\
Claims&& &Binomial& & & &$a_0=a_1=1$\\\hline
0&7840&7635.46&7718.056&7735.846&7730.347&7797.46&7757.174\\\hline
1&1317&1636.852&1494.167&1463.028&1455.685&1343.856&1428.108\\\hline
2&239&175.4500&216.9461&225.6458&240.053&271.7930&233.7039\\\hline
3&42&12.53737&27.99961&31.65992&31.0723&41.80549&35.85438\\\hline
4&14&0.6719245&3.387843&4.205013&3.457364&5.389682&5.280678\\\hline
5&4&0.02880876&0.3935191&0.5388292&0.3491404&0.6221004&0.7561408\\\hline
6&4&0.001029313&0.04443999&0.06732125&0.03299364&0.06667545&0.1060622\\\hline
7&1&3.152271e-05&0.004916174&0.008254485&0.00297182&0.006779088&0.01464467 \\\hline
Negative&-&5490.781&5388.843&5377.51&5384.057&5376.699&5367.193 \\
log-likelihood&&&&&&&\\\hline
Observed&-&48229.53&651.966&414.0822&900.412&790.0631&248.2751\\
Chisquare&&&&&&&\\\hline
\end{tabular}
}
\end{center}

\newpage
\begin{center}{\bf Table:4 Data Set-IV}\\
\vspace{0.2 in}
{
\footnotesize 
\begin{tabular}{|c|c|c|c|c|c|c|c|}
\hline
Number&Observed&Fitted& Fitted&Fitted&Fitted& Fitted& Fitted\\
of&&Poisson &Negative&dLindley &dxgamma-I &dxgamma-II&NDOPPE\\
Claims&& &Binomial& & & &$a_0=a_1=1$\\\hline
0&3719&3668.600&3675.159&3676.208&3671.561&3676.639&3678.629\\\hline
1&232&317.2765&304.7798&302.4305&306.8567&296.5420&298.3138\\\hline
2&38&13.71973&18.95647&20.08010&20.50728&25.19934&21.50345\\\hline
3&7&0.3955141&1.048036&1.209059&1.029554&1.538505&1.453163\\\hline
4&3&0.008551448&0.05432081&0.06881941&0.0438056&0.07764103&0.09427396\\\hline
5&1&0.0001479133&0.002702885&0.003776502&0.001681779&0.003491711&0.005946133\\\hline
Negative&-&1246.077&1221.197&1217.698&1221.520&1211.224&1213.141\\
log-likelihood&&&&&&&\\\hline
Observed&-&7982.045&598.55&448.2736&860.5252&777.0424&304.7557\\
Chisquare&&&&&&&\\\hline
\end{tabular}
}
\end{center}

\begin{center}{\bf Table:5 Data Set-V}\\
\vspace{0.2 in}
{
\footnotesize 
\begin{tabular}{|c|c|c|c|c|c|c|c|}
\hline
Number&Observed&Fitted& Fitted&Fitted&Fitted& Fitted& Fitted\\
of&&Poisson &Negative&dLindley &dxgamma-I &dxgamma-II&NDOPPE\\
Claims&& &Binomial& & & &$a_0=1,a_1=3.35$\\\hline
0&96978&96688.27&96929.48&96981.95& 96963.48&97185.29&96981.69\\\hline
1&9240&9774.58&9325.676&9228.955&9250.88&8858.42&9227.194\\\hline
2&704&494.0744&672.924&710.0378&715.6196&865.322&711.7372\\\hline
3&43&16.64928&43.16177&49.56523&41.84904&61.1897&49.85381\\\hline
4&9& 0.4207845&2.595396 &3.271535&2.078198&3.582792&3.303031\\\hline
Negative&-&36188.25&36106.19&36104.22&36106.61&36130.61&36104.22\\
log-likelihood&&&&&&&\\\hline
Observed&-&335.9228&18.05162 &10.96488&23.28958&210.8106&10.87023\\
Chisquare&&&&&&&\\\hline
\end{tabular}
}
\end{center}

\begin{center}{\bf Table:6 Data Set-VI}\\
\vspace{0.2 in}
{
\footnotesize 
\begin{tabular}{|c|c|c|c|c|c|c|c|}
\hline
Number&Observed&Fitted& Fitted&Fitted&Fitted& Fitted& Fitted\\
of&&Poisson &Negative&dLindley &dxgamma-I &dxgamma-II&NDOPPE\\
Claims&& &Binomial& & & &$a_0=a_1=1$\\\hline
0&20592&20417.77&20522.27&20544.93&20540.09&20632.49&20572.3\\\hline
1&2651&2947.815&2760.887&2720.243&2718.655&2560.743&2671.185\\\hline
2&297&212.7954&278.5692&292.3791&302.3779&355.8121&308.2993\\\hline
3&41&10.24078&24.98418&28.54859&25.85673&36.56728&33.35894\\\hline
4&7&0.3696281&2.100720&2.637138&1.887721& 3.127832&3.46516\\\hline
5&0&0.01067301&0.1695675&0.2349461&0.1247285&0.2389654&0.3499451\\\hline
6&1&0.0002568194&0.01330708&0.02040511&0.007701431&0.01693577&0.03461957\\\hline
Negative&-&10297.85&10233.72&10228.45&10232.25&10250.72&10224.71\\
log-likelihood&&&&&&&\\\hline
Observed&-&4167.816&100.8537&61.85467&152.6027&176.145 &33.21162\\
Chisquare&&&&&&&\\\hline
\end{tabular}
}
\end{center}
\newpage
\begin{center}{\bf Table:7 Data Set-VII}\\
\vspace{0.2 in}
{
\footnotesize 
\begin{tabular}{|c|c|c|c|c|c|c|c|}
\hline
Number&Observed&Fitted& Fitted&Fitted&Fitted& Fitted& Fitted\\
of&&Poisson &Negative&dLindley &dxgamma-I &dxgamma-II&NDOPPE\\
Claims&& &Binomial& & & &$a_0=1,a_1=a_2=0.01$\\\hline
0&71087&67424.99&67960.82&68052.5&67775.91&68179.48&68472.88\\\hline
1&6744&12363.00&11415.29&11225.99&11366.38&10612.39&10556.37\\\hline
2&2067&1133.436&1438.059&1507.513&1646.199&1910.340&1652.717\\\hline
3&690&69.27539&161.0327&184.0637&185.8067&259.1091&262.0185\\\hline
4&248&3.175573&16.90529&21.26899&17.98504&29.38657&41.94244\\\hline
5&95&0.1164542&1.703736&2.370814&1.578403&2.981471&6.760695\\\hline
6&34&0.003558829&0.1669351&0.2576519&0.1295657&0.2807841&1.09478\\\hline
7&17&9.322066e-05&0.01602279&0.02747802&0.01013466&0.02508084&0.1777548\\\hline
8&4&2.136610e-06&0.001513872&0.002888127&0.0007644298&0.00215381&0.02889426\\\hline
9&3&4.352973e-08&0.0001412684&0.0003000664&5.604773e-05&0.0001794176&0.004696679\\\hline
10&3&7.981583e-10&1.305077e-05&3.088308e-05&4.017203e-06&1.458903e-05&0.0007627534\\\hline
11&2&1.330453e-11&1.195703e-06&3.153703e-06&2.826347e-07&1.163190e-06&0.0001236865\\\hline
Negative&-&44481.26&42392.02&42097.6&42256.75& 41607.45&41257.56 \\
log-likelihood&&&&&&&\\\hline
Observed&-&312143246723&4146376&1619743&16623717&14249230&53609.78\\
Chisquare&&&&&&&\\\hline
\end{tabular}
}
\end{center}
\begin{center}{\bf Table:8 Data Set-VIII}\\
\vspace{0.2 in}
{
\footnotesize 
\begin{tabular}{|c|c|c|c|c|c|c|c|}
\hline
Number&Observed&Fitted& Fitted&Fitted&Fitted& Fitted& Fitted\\
of&&Poisson &Negative&dLindley &dxgamma-I &dxgamma-II&NDOPPE\\
Claims&& &Binomial& & & &$a_0=a_1=1$\\\hline
0&530642&528917.3&529526.7&529646.2&529458.2&530005.8&529832.4\\\hline
1&33495&36734.96&35556.89&35318.90&35557.22&34536.13&34956.52\\\hline
2&2585&1275.679&1790.692&1896.219&1867.180&2302.655&2050.054\\\hline
3&211&29.53329&80.16146&92.25659&72.9354&108.9847&112.713\\\hline
4&25&0.5127949&3.364198&4.241825&2.40716&4.252032&112.713\\\hline
Negative&-&146704.8&146051.2&145973.0&146039.8&146083.2&145879.5\\
log-likelihood&&&&&&&\\\hline
Observed&-&3919.575&826.9477&600.6708&871.6106 &946.4252&348.6467\\
Chisquare&&&&&&&\\\hline
\end{tabular}
}
\end{center}


\section*{Bibliography}
\begin{enumerate}


\item Ahmed, A-H. N. and Afify, A. Z. (2019): A new discrete analog of the continuous Lindley distribution, with reliability applications, Proceedings of 62nd ISI World Statistics Congress, Kuala Lumpur, vol.5, 7-14.(isi2019.org/proceeding/3.CPS/CPS VOL 5)
\item Barlow, R. E. and Proschan, F. (1965): Mathematical theory of reliability. John Wiley $\&$ Sons Inc., New York.
\item Bakouch, H.S., Jazi, M. A. and Nadarajah, S. (2014): A new discrete distribution, Statistics, 48, 200-240.
\item Berm$\acute{u}$dez, L. (2009): A priori ratemaking using bivariate Poisson regression models, Insurance: Mathematics and Economics, 44, 135-141.
\item Bouchahed, L. and Zeghdoudi, H. (2018): A new and unified approach in generalizing the Lindley's distribution with applications, Statistics in Transition, 19(1), 61-74.
\item Boucher, J. P., Denuit, M. and M. Guill$\acute{e}$n, M. (2007): Risk classification for claim counts: A comparative analysis of various
zero-inflated mixed Poisson and hurdle models, North American Actuarial Journal, 11(4), 110-131.
\item Famoye, F. (1993): Restricted generalized Poisson regression model, Communications in Statistics-Theory and Methods, 22, 1335-1354.
\item G$\acute{o}$mez-D$\acute{e}$niz, E. (2010): Another generalization of the geometric distribution, Test, 19, 399-415.
\item G$\acute{o}$mez-D$\acute{e}$niz, E. and Calder$\acute{i}$n-Ojeda, E. (2011): The discrete Lindley distribution: properties and applications, Journal of Statistical Computation and Simulation, 81, 1405-1416.
\item Jazi, M. A., Lai, C. D. and Alamatsaz, M. H. (2010): A discrete inverse Weibull distribution and estimation of its parameters, Statistical Methodology, 7, 121-132.
\item Johnson, N., Kotz, S. and Kemp, A. (1992): Univariate discrete distributions. 2nd edn, John Wiley and Sons, New York.
\item Krishna, H. and Pundir, P. S. (2009): Discrete Burr and discrete Pareto distributions, Statistical Methodology, 6, 177-188.
\item Lawless, J. F. (2003):Statistical models and methods for lifetime data. 2nd ed., John Wiley and Sons, New York.
\item Lindley, D. V. (1958): Fiducial distributions and Bayes' theorem, Journal of Royal Statistical Society, B20, 102-107.
\item Maiti, S. S., Dey, M. and Sarkar(Mondal), S. (2018): Discrete xgamma distributions: properties,estimation and an application to the collective risk model, Journal of Reliability and Statistical Studies, 11(1), 117-132.
\item Nakagawa, T. and Osaki, S. (1975): Discrete Weibull distribution, IEEE Transactions on Reliability, 24, 300-301.
\item Ghitany, M. E., Atieh, B. and Nadarajah, S. (2008): Lindley distribution and its application, Mathematics and Computers in Simulation,78(4), 493-506.
\item Rolski, T., Schmidli, H., Schmidt, V. and Teugel, J. (1999): Stochastic processes for insurance and finance. John Wiley $\&$ Sons.
\item Roy, D. (2002): Discretization of continuous distributions with application to stress-strength reliability, Calcutta Statistical Association Bulletin, 52, 297-313.
\item Roy, D. (2003): The discrete normal distribution, Communications in Statistics-Theory and Methods, 32, 1871-1883.
\item Roy, D. (2004): Discrete Rayleigh distribution, IEEE Transactions on Reliability, 53, 255-260.
\item Sinha, S.K. (1986): Reliability and Life testing. Wiley Eastern Ltd., New Delhi.
\item Stein, W. E. and Dattero, R. (1984): A new discrete Weibull distribution, IEEE Transactions on Reliability, 33, 196-197.
\end{enumerate}
\end{document}